\documentclass[12pt,reqno]{amsart}

\usepackage{hyperref}
\usepackage{amsthm}
\usepackage[letterpaper]{geometry}
\usepackage{geometry}

\newtheorem{theorem}{Theorem}[section]
\newtheorem{proposition}[theorem]{Proposition}

\newtheorem{lemma}[theorem]{Lemma}

\newtheorem{corollary}[theorem]{Corollary}

\theoremstyle{remark}

\theoremstyle{definition}
\newtheorem{definition}[theorem]{Definition}
\newtheorem{remark}[theorem]{Remark}

\numberwithin{equation}{section}
\numberwithin{figure}{section}
\numberwithin{table}{section}

\newcommand{\arccot}{\mathrm{arccot}}

\allowdisplaybreaks[3]

\begin{document}

 \title{Supercritical deformed Hermitian-Yang-Mills equation}

 \author{Gao Chen} \thanks{Support for this research was provided by the Office of the Vice Chancellor for Research and Graduate Education at the University of Wisconsin-Madison with funding from the Wisconsin Alumni Research Foundation.}
  \address{Department of Mathematics, University of Wisconsin-Madison, Madison, WI, 53706}
 \email{gchen233@wisc.edu}

\begin{abstract}
In this paper, we provide a necessary and sufficient condition for the solvability of the supercritical deformed Hermitian-Yang-Mills equation using integrals on subvarieties. This result confirms the mirror version of the Thomas-Yau conjecture about special Lagrangian submanifolds on Calabi-Yau manifolds.
\end{abstract}

\date{}

\maketitle

\setcounter{tocdepth}{1}
\tableofcontents

\section{Introduction}

In this paper, we study the deformed Hermitian-Yang-Mills equation using integrals on subvarieties. The main motivation for studying deformed Hermitian-Yang-Mills equation is to study special Lagrangian submanifolds on Calabi-Yau manifolds. Special Lagrangian submanifolds were introduced by Harvey-Lawson \cite{HarveyLawson} and play an important role in symplectic geometry, mathematical physics and the study of minimal surfaces. On the Euclidean space, Harvey-Lawson \cite{HarveyLawson} proved that the special Lagrangian equation can be written as
\[\mathrm{Im} (e^{-\sqrt{-1}\hat \theta}\det (I+\sqrt{-1}\mathrm{Hess} F)) = 0 \] for a constant $\hat\theta$ and a function $F$. On a K\"ahler manifold $M$ (not necessarily Calabi-Yau) with a K\"ahler class $\chi$ and a real closed (1,1)-form $\omega_0$, a natural analogy of this equation is the deformed Hermitian-Yang-Mills equation
\[\mathrm{Im} (e^{-\sqrt{-1}\hat \theta}(\chi+\sqrt{-1}\omega_{\varphi})^n) = 0, \]
where $\omega_\varphi=\omega_0+\sqrt{-1}\partial\bar\partial\varphi$.
In fact, the deformed Hermitian-Yang-Mills equation was derived by Mari\~{n}o, Minasian, Moore, and Strominger \cite{MarinoMinasianMooreStrominger} and Leung-Yau-Zaslow \cite{LeungYauZaslow} as the mirror equation of the special Lagrangian equation.

At each point $p$, there exist local coordinates such that $\omega_\varphi=\sum_{i=1}^{n}\sqrt{-1}\lambda_{i}dz^i\wedge d\bar z^i$ and $\chi=\sum_{i=1}^{n}\sqrt{-1}dz^i\wedge d\bar z^i$ at $p$. Then the deformed Hermitian-Yang-Mills equation can be written as
\[\sum_{i=1}^{n}\arccot(\lambda_i)=\theta_0,\]
where $\arccot$ is the inverse function of $\cot$ with range $(0,\pi)$ and $\theta_0$ is a real number such that $\theta_0\equiv \frac{n\pi}{2}-\hat\theta \mod \pi$. When $0<\theta_0<\pi$, it is called supercritical. The lower bound $\cot(\theta_0)\chi$ for $\omega_\varphi$ exists if and only if the solution to the deformed Hermitian-Yang-Mills equation is supercritical. Therefore, the supercritical case is much more important than other cases.

A celebrated conjecture by Thomas-Yau \cite{Thomas, ThomasYau} is that the existence of special Lagrangian submanifolds is equivalent to stability. On the mirror side, the stability condition was studied in a series of works \cite{CollinsJacobYau, CollinsYau, ChuCollinsLee}. See also \cite{CollinsShi} for a survey. The stability condition for deformed Hermitian-Yang-Mills equation can be understood as a condition on integrals on subvarieties. In this paper, we prove the equivalence of the stability with the solvability of deformed Hermitian-Yang-Mills equation in the supercritical case. It confirms the mirror version of Thomas-Yau conjecture.

\begin{theorem}
(Main Theorem) Fix a K\"ahler manifold $M^n$ with a K\"ahler metric $\chi$ and a real closed (1,1)-form $\omega_0$. Assume that there exists a constant $\theta_0\in(0,\pi)$ such that
\[\int_M(\mathrm{Re}(\omega_\varphi+\sqrt{-1}\chi)^n-\cot(\theta_0)\mathrm{Im}(\omega_\varphi+\sqrt{-1}\chi)^n)=0,\]
then the followings are equivalent:

(1) There exists a smooth function $\varphi$ such that the corresponding eigenvalues $\lambda_i$ satisfy the deformed-Hermitian-Yang-Mills equation
\[\sum_{i=1}^{n}\arccot(\lambda_i)=\theta_0.\]

(2) For any smooth test family $\omega_{t,0}$, there exists a constant $\epsilon_1>0$ independent of $t, V$ such that for any $t\ge 0$ and any $p$-dimensional subvariety $V$,

\[\int_V(\mathrm{Re}(\omega_{t,0}+\sqrt{-1}\chi)^p-\cot(\theta_0)\mathrm{Im}(\omega_{t,0}+\sqrt{-1}\chi)^p)\ge (n-p)\epsilon_1\int_V \chi^p.\]

(3) There exist a test family $\omega_{t,0}$ and a constant $\epsilon_1>0$ independent of $t, V$ such that for any $t\ge 0$ and any $p$-dimensional subvariety $V$,

\[\int_V(\mathrm{Re}(\omega_{t,0}+\sqrt{-1}\chi)^p-\cot(\theta_0)\mathrm{Im}(\omega_{t,0}+\sqrt{-1}\chi)^p)\ge (n-p)\epsilon_1\int_V \chi^p.\]

Here, we call a smooth family $\omega_{t,0}, t\in[0,\infty)$ of real closed (1,1)-forms as a test family if and only if all of the following conditions hold:

(A) When $t=0$, $\omega_{t,0}=\omega_0$.

(B) For all $s>t$, $\omega_s-\omega_t$ is positive definite.

(C) There exists $T\ge 0$ such that for all $t\ge T$, $\omega_{t,0}-\cot(\frac{\theta_0}{n})\chi$ is positive definite.
\label{Main-theorem}
\end{theorem}

\begin{remark}
The main reason for introducing the test family is to correctly choose the branch of the $\cot$ function. In the special case when $\omega_0>0$, an important choice of the test family is $\omega_{t,0}=t\omega_0$. In general cases, $\omega_{t,0}=\omega_0+t\chi$ is always a test family. However, more choices of test families are allowed.
\end{remark}

The main strategy to prove the main theorem is the same as the article \cite{Chen}. In fact, the $\theta_0\in(0,\frac{\pi}{4})$ case of the main theorem has already been proved in \cite{Chen}. The main reason for the restriction of the range of $\theta_0$ in that paper was because the function $\sum_{i=1}^{n}\arccot(\lambda_i)$ is no longer convex for $\theta_0>\frac{\pi}{2}$ and is not convex enough for $\theta_0\in[\frac{\pi}{4},\frac{\pi}{2}]$. The key observation of this paper is that the function $-\cot(\sum_{i=1}^{n}\arccot(\lambda_i))$ instead is convex on the whole range $\theta_0\in (0,\pi)$. See also \cite{HarveyLawson2020} and \cite{Takahashi} for slightly different observations.

For simplicity, we define the following notations:

\begin{definition}
Define $P:\mathbb{R}^n\to (0,(n-1)\pi)$ and $Q:\mathbb{R}^n\to (0,n\pi)$ by
\[P(\lambda_1,...\lambda_n)=\max_{i=1}^{n}(\sum_{k\not=i}\arccot(\lambda_k)),\]
and \[Q(\lambda_1,...\lambda_n)=\sum_{k=1}^{n}\arccot(\lambda_k).\]
Let $0<\theta_0<\Theta_0<\pi$ be any constants. Define $\Gamma_{\theta_0,\Theta_0}$ be the subset of $\mathbb{R}^n$ such that $P$ is smaller than $\theta_0$ and $Q$ is smaller than $\Theta_0$. Its closure is denoted by $\bar\Gamma_{\theta_0,\Theta_0}$.

Let $A, B$ be Hermitian matrices. Assume that $A$ is positive definite. Then $P_A(B)$ is defined as $P(\lambda_1,...\lambda_n)$, where $\lambda_i$ are the eigenvalues of the matrix $A^{-1}B$. The function $Q_A(B)$ is defined as $Q(\lambda_1,...\lambda_n)$. The set $\Gamma_{A,\theta_0,\Theta_0}$ is defined as the set of all matrices $B$ such that $(\lambda_1,...,\lambda_n)\in \Gamma_{\theta_0,\Theta_0}$. Its closure is denoted by $\bar\Gamma_{A,\theta_0,\Theta_0}$.
\end{definition}

The more suitable version for induction is the following:

\begin{proposition}
Fix a K\"ahler manifold $M^n$ with a K\"ahler metric $\chi$ and a real closed (1,1)-form $\omega_0$. Let $\theta_0\in(0,\pi)$ be a constant and let $\Theta_0\in(\theta_0,\pi)$ be another constant. Then there exists a constant $\epsilon_2>0$ only depending on $n, \theta_0, \Theta_0$ such that the following holds:

Assume that (1) When $n\ge 4$, $f>-\epsilon_2$ is a smooth function satisfying \[\int_M f\chi^n=\int_M(\mathrm{Re}(\omega_\varphi+\sqrt{-1}\chi)^n-\cot(\theta_0)\mathrm{Im}(\omega_\varphi+\sqrt{-1}\chi)^n)\ge 0;\]

(2) When $n=1,2,3$, $f\ge 0$ is a constant satisfying \[\int_M f\chi^n=\int_M(\mathrm{Re}(\omega_\varphi+\sqrt{-1}\chi)^n-\cot(\theta_0)\mathrm{Im}(\omega_\varphi+\sqrt{-1}\chi)^n)\ge 0;\]

(3) There exists a test family $\omega_{t,0}$ and a constant $\epsilon_1>0$ independent of $t, V$ such that for any $t\ge 0$ and any $p$-dimensional subvariety $V$,

\[\int_V(\mathrm{Re}(\omega_{t,0}+\sqrt{-1}\chi)^p-\cot(\theta_0)\mathrm{Im}(\omega_{t,0}+\sqrt{-1}\chi)^p)\ge (n-p)\epsilon_1\int_V \chi^p.\]

Then there exists a smooth function $\varphi$ satisfying
\[\mathrm{Re}(\omega_\varphi+\sqrt{-1}\chi)^n-\cot(\theta_0)\mathrm{Im}(\omega_\varphi+\sqrt{-1}\chi)^n-f\chi^n=0,\] and $\omega_\varphi=\omega_0+\sqrt{-1}\partial\bar\partial\varphi\in \Gamma_{\chi,\theta_0,\Theta_0}$.

\label{Main-induction}
\end{proposition}

\begin{remark}
Proposition \ref{Main-induction} is similar to Theorem 5.1 of \cite{Chen} and the results in \cite{Pingali}.
\end{remark}

In Step 1, we need to prove the following proposition:

\begin{proposition}
Fix a K\"ahler manifold $M^n$ with a K\"ahler metric $\chi$ and a real closed (1,1)-form $\omega_0$. Let $\theta_0\in(0,\pi)$ be a constant and let $\Theta_0\in(\theta_0,\pi)$ be another constant. Then there exists a constant $\epsilon_2>0$ only depending on $n, \theta_0, \Theta_0$ such that the following holds:

Assume that (1) When $n\ge 4$, $f>-\epsilon_2$ is a smooth function satisfying \[\int_M f\chi^n=\int_M(\mathrm{Re}(\omega_\varphi+\sqrt{-1}\chi)^n-\cot(\theta_0)\mathrm{Im}(\omega_\varphi+\sqrt{-1}\chi)^n)\ge 0;\]

(2) When $n=1,2,3$, $f\ge 0$ is a constant satisfying \[\int_M f\chi^n=\int_M(\mathrm{Re}(\omega_\varphi+\sqrt{-1}\chi)^n-\cot(\theta_0)\mathrm{Im}(\omega_\varphi+\sqrt{-1}\chi)^n)\ge 0;\]

(3) $\omega_0\in \Gamma_{\chi,\theta_0,\Theta_0}$.

Then there exists a smooth function $\varphi$ satisfying
\[\mathrm{Re}(\omega_\varphi+\sqrt{-1}\chi)^n-\cot(\theta_0)\mathrm{Im}(\omega_\varphi+\sqrt{-1}\chi)^n-f\chi^n=0,\] and $\omega_\varphi=\omega_0+\sqrt{-1}\partial\bar\partial\varphi\in\Gamma_{\chi,\theta_0,\Theta_0}$.
\label{Solvability-assuming-sub-solution}
\end{proposition}

\begin{remark}
Proposition \ref{Solvability-assuming-sub-solution} is similar to Theorem 5.2 of \cite{Chen} and the results in \cite{CollinsJacobYau}.
\end{remark}

In Step 2, we need to prove the following:

\begin{proposition}
Fix a K\"ahler manifold $M^n$ with a K\"ahler metric $\chi$ and a test family $\omega_{t,0}$ of real closed (1,1)-forms. Suppose that for all $t>0$, there exist a constant $c_t>0$ and a smooth function $\varphi_t$ such that $\omega_t=\omega_{t,0}+\sqrt{-1}\partial\bar\partial\varphi_t$ satisfies \[\mathrm{Re}(\omega_t+\sqrt{-1}\chi)^n-\cot(\theta_0)\mathrm{Im}(\omega_t+\sqrt{-1}\chi)^n-c_t\chi^n=0,\] and $\omega_t\in \Gamma_{\chi,\theta_0,\Theta_0}$. Then there exist a constant $\epsilon_3>0$ and a current $\omega_4\in[\omega_0-\epsilon_3\chi]$ such that $\omega_4\in \bar\Gamma_{\chi,\theta_0,\Theta_0}$ in the sense of Definition \ref{Definition-current-sub-solution}.
\label{Current-sub-solution}
\end{proposition}

\begin{remark}
Proposition \ref{Current-sub-solution} is similar to Theorem 5.4 of \cite{Chen}.
\end{remark}

In Step 3, we need to do the regularization similar to Section 4 of \cite{Chen}.

In Section 2, we will prove Proposition \ref{Solvability-assuming-sub-solution} by proving the preliminary conditions to apply the results in \cite{Szekelyhidi}. In Section 3, we will provide Definition \ref{Definition-current-sub-solution} and prove Proposition \ref{Current-sub-solution}. In Section 4, we finish the proof of Proposition \ref{Main-induction} and Theorem \ref{Main-theorem}.

\section{The analysis part}

In this section, we prove Proposition \ref{Solvability-assuming-sub-solution} by proving the preliminary conditions to apply the results in \cite{Szekelyhidi}. The key lemma is the following:

\begin{lemma}
For any $0<\theta_0<\Theta_0<\pi$, there exist a constant $\epsilon_2>0$ only depending on $n, \theta_0, \Theta_0$ and constants $\epsilon_5>0, C_6>0$ only depending on $n, \Theta_0$ such that the following holds:

Assume that $f$ is a parameter such that $f \ge 0$ when $n=1,2,3$, and $f\ge -\epsilon_2$ when $n\ge 4$. Then the function $F:\bar\Gamma_{\theta_0,\Theta_0}\to\mathbb{R}$ defined by
\[F(\lambda_1,...\lambda_n) =\frac{\mathrm{Re}\prod_{k=1}^{n}(\lambda_k+\sqrt{-1})}{\mathrm{Im}\prod_{k=1}^{n}(\lambda_k+\sqrt{-1})}
-\frac{f}{\mathrm{Im}\prod_{k=1}^{n}(\lambda_k+\sqrt{-1})}-\cot(\theta_0)\] satisfies the following properties:

(1) $\frac{1}{\mathrm{Im}\prod_{k=1}^{n}(\lambda_k+\sqrt{-1})}\le \frac{1}{C_6}$;

(2) \[|\frac{\partial}{\partial\lambda_i}\frac{1}{\mathrm{Im}\prod_{k=1}^{n}(\lambda_k+\sqrt{-1})}|\le \frac{1}{\sqrt{C_6}}\sqrt{\frac{\prod_{k=1}^{n}(1+\lambda_k^2)}{(\mathrm{Im}\prod_{k=1}^{n}(\lambda_k+\sqrt{-1}))^3}} \frac{1}{\sqrt{1+\lambda_i^2}};\]

(3) $\frac{\partial F}{\partial\lambda_i}>0$;

(4) When $n\ge 4$, for any real numbers $u_i$,
\[\sum_{i,j=1}^{n}\frac{\partial^2 F}{\partial\lambda_i\partial\lambda_j}u_iu_j\le -\epsilon_5\frac{\prod_{k=1}^{n}(1+\lambda_k^2)}{(\mathrm{Im}\prod_{k=1}^{n}(\lambda_k+\sqrt{-1}))^3} \sum_{i=1}^{n}\frac{u_i^2}{1+\lambda_i^2};\]

When $n=1,2,3$, for any real numbers $u_i$, $\sum_{i,j=1}^{n}\frac{\partial^2 F}{\partial\lambda_i\partial\lambda_j}u_iu_j\le 0$;

(5) If $\lambda\in\bar\Gamma_{\theta_0,\Theta_0}$, and $F(\lambda)=0$, then $\lambda\in\Gamma_{\theta_0,\Theta_0}$;

(6) For any $\lambda\in\Gamma_{\theta_0,\Theta_0}$, the set
\[\{\lambda'\in\Gamma_{\theta_0,\Theta_0}: F(\lambda')=0, \lambda'_i\ge\lambda_i, \ \text{for all}\ i=1,2,3,...,n\}\] is bounded, where the bound depends on $n,\theta_0,\Theta_0,\lambda, |f|$;

(7) $\bar\Gamma_{\theta_0,\Theta_0}$ is convex;

(8) $\frac{\partial}{\partial\lambda_i} F(\lambda)\le\frac{\partial}{\partial\lambda_j} F(\lambda)$ if $\lambda_i\ge\lambda_j$;

(9) For any positive definite Hermitian matrix $A$, the function $F_A:\bar\Gamma_{A,\theta_0,\Theta_0}\to\mathbb{R}^n$ is concave, where $F_A(B)=F(\lambda_1,...\lambda_n)$ if $\lambda_i$ are the eigenvalues of $A^{-1}B$;

(10) For any positive definite Hermitian matrix $A$, the set $\bar\Gamma_{A,\theta_0,\Theta_0}$ is convex.

\label{Property-of-F}
\end{lemma}

\begin{proof}
When $n=1$, $F(\lambda_1)=\lambda_1-f-\cot(\theta_0)$. So all the properties are trivial. So we assume that $n\ge 2$.

For simplicity, define $\theta_i=\arccot(\lambda_i)$. Then it is easy to see that \[\sin(\theta_i)=\frac{1}{\sqrt{1+\lambda_i^2}}, \cos(\theta_i)=\frac{\lambda_i}{\sqrt{1+\lambda_i^2}},\]
\[\mathrm{Re}\prod_{k=1}^{n}(\lambda_k+\sqrt{-1})=\cos(\sum_{k=1}^{n}\theta_k)\prod_{k=1}^{n}\sqrt{1+\lambda_k^2},\]
and
\[\mathrm{Im}\prod_{k=1}^{n}(\lambda_k+\sqrt{-1})=\sin(\sum_{k=1}^{n}\theta_k)\prod_{k=1}^{n}\sqrt{1+\lambda_k^2}.\]

(1) First of all, there exists $C_7>0$ only depending on $\Theta_0$ such that $\sin(x)\ge C_7$ as long as
$\pi>\Theta_0\ge x\ge \frac{\Theta_0}{2}>0$. Moreover, it is easy to see that there exist constants $C_8>0, C_9>0$ only depending on $\Theta_0$ such that
$\sin(x)\ge C_8 x$ for all $x\in(0,\Theta_0)$ and $\tan(x)\le C_9 x$ for all $x\in(0,\frac{\Theta_0}{2})$.

Now we study two cases. If $\sum_{k=1}^{n}\theta_k\ge \frac{\Theta_0}{2}$, then
\[\mathrm{Im}\prod_{k=1}^{n}(\lambda_k+\sqrt{-1})=\sin(\sum_{k=1}^{n}\theta_k)\prod_{k=1}^{n}\sqrt{1+\lambda_k^2}
\ge C_7.\]
If $\sum_{k=1}^{n}\theta_k\le \frac{\Theta_0}{2}$, then $\lambda_i \ge \cot(\frac{\Theta_0}{2})>0$ for all $i=1,2,3,...,n$. So
\[\mathrm{Im}\prod_{k=1}^{n}(\lambda_k+\sqrt{-1})\ge C_8 (\sum_{k=1}^{n}\theta_k)\prod_{k=1}^{n}\lambda_k\ge \frac{C_8}{C_9}(\sum_{k=1}^{n}\frac{1}{\lambda_k})\prod_{k=1}^{n}\lambda_k\ge \frac{nC_8}{C_9}\cot^{n-1}(\frac{\Theta_0}{2}).\]
So we can choose $C_6$ as $\min\{\frac{nC_8}{C_9}\cot^{n-1}(\frac{\Theta_0}{2}), C_7\}$.

(2)

\[\begin{split}
|\frac{\partial}{\partial\lambda_i}\frac{1}{\mathrm{Im}\prod_{k=1}^{n}(\lambda_k+\sqrt{-1})}| =\frac{\mathrm{Im}\prod_{k\not=i}(\lambda_k+\sqrt{-1})}{(\mathrm{Im}\prod_{k=1}^{n}(\lambda_k+\sqrt{-1}))^2}
\le\frac{\sqrt{\prod_{k\not=i}(1+\lambda_k^2)}}{(\mathrm{Im}\prod_{k=1}^{n}(\lambda_k+\sqrt{-1}))^2}\\
=\frac{\sqrt{\prod_{k=1}(1+\lambda_k^2)}}{(\mathrm{Im}\prod_{k=1}^{n}(\lambda_k+\sqrt{-1}))^2} \frac{1}{\sqrt{1+\lambda_i^2}}
\le
\frac{1}{\sqrt{C_6}}\sqrt{\frac{\prod_{k=1}^{n}(1+\lambda_k^2)}{(\mathrm{Im}\prod_{k=1}^{n}(\lambda_k+\sqrt{-1}))^3}} \frac{1}{\sqrt{1+\lambda_i^2}}.
\end{split}\]

(3) \[
\begin{split}
\frac{\partial F}{\partial\lambda_i}&=\frac{1}{\sin^2(\sum_{k=1}^{n}\theta_k)}\frac{1}{1+\lambda_i^2}
+\frac{f \mathrm{Im}\prod_{k\not=i}(\lambda_k+\sqrt{-1})}{(\mathrm{Im}\prod_{k=1}^{n}(\lambda_k+\sqrt{-1}))^2}\\
&=\frac{1}{\sin^2(\sum_{k=1}^{n}\theta_k)}\frac{1}{1+\lambda_i^2}
+\frac{1}{\sin^2(\sum_{k=1}^{n}\theta_k)}\frac{1}{1+\lambda_i^2}\frac{f \sin (\sum_{k\not=i}\theta_k)}{\prod_{k\not=i} \sqrt{1+\lambda_k^2}}.
\end{split}\]
Therefore, if $\epsilon_2<1$, then $\frac{\partial F}{\partial\lambda_i}>0$.

(4)
\[
\begin{split}
\frac{\partial^2 F}{\partial\lambda_i^2}&=\frac{2\cos(\sum_{k=1}^{n}\theta_k)}{\sin^3(\sum_{k=1}^{n}\theta_k)}\frac{1}{(1+\lambda_i^2)^2}
+\frac{1}{\sin^2(\sum_{k=1}^{n}\theta_k)}\frac{-2\lambda_i}{(1+\lambda_i^2)^2}-2\frac{f (\mathrm{Im}\prod_{k\not=i}(\lambda_k+\sqrt{-1}))^2}{(\mathrm{Im}\prod_{k=1}^{n}(\lambda_k+\sqrt{-1}))^3}\\
&=\frac{2\cot(\sum_{k=1}^{n}\theta_k)-2\lambda_i}{\sin^2(\sum_{k=1}^{n}\theta_k)(1+\lambda_i^2)^2}
-\frac{f \prod_{k=1}^{n}(1+\lambda_k^2)}{(1+\lambda_i^2)(\mathrm{Im}\prod_{k=1}^{n}(\lambda_k+\sqrt{-1}))^3} \cdot 2\sin^2(\sum_{k\not=i}\theta_k).
\end{split}\]

When $i\not=j$, then

\[
\begin{split}
\frac{\partial^2 F}{\partial\lambda_i\partial\lambda_j}
&=\frac{2\cot(\sum_{k=1}^{n}\theta_k)}{\sin^2(\sum_{k=1}^{n}\theta_k)}\frac{1}{(1+\lambda_i^2)(1+\lambda_j^2)}\\
&+\frac{f (\mathrm{Im}\prod_{k\not=i,j}(\lambda_k+\sqrt{-1}))(\mathrm{Im}\prod_{k=1}^{n}(\lambda_k+\sqrt{-1}))}
{(\mathrm{Im}\prod_{k=1}^{n}(\lambda_k+\sqrt{-1}))^3}\\
&-\frac{2f (\mathrm{Im}\prod_{k\not=i}(\lambda_k+\sqrt{-1}))
(\mathrm{Im}\prod_{k\not=j}(\lambda_k+\sqrt{-1}))}{(\mathrm{Im}\prod_{k=1}^{n}(\lambda_k+\sqrt{-1}))^3}.\\
\end{split}\]

Using
\[\mathrm{Im}\prod_{k\not=j}(\lambda_k+\sqrt{-1})=\lambda_i\mathrm{Im}\prod_{k\not=i,j}(\lambda_k+\sqrt{-1})
+\mathrm{Re}\prod_{k\not=i,j}(\lambda_k+\sqrt{-1}),\]
\[\mathrm{Im}\prod_{k\not=i}(\lambda_k+\sqrt{-1})=\lambda_j\mathrm{Im}\prod_{k\not=i,j}(\lambda_k+\sqrt{-1})
+\mathrm{Re}\prod_{k\not=i,j}(\lambda_k+\sqrt{-1}),\]
\[\mathrm{Im}\prod_{k=1}^n(\lambda_k+\sqrt{-1})=(\lambda_i\lambda_j-1)\mathrm{Im}\prod_{k\not=i,j}(\lambda_k+\sqrt{-1})
+(\lambda_i+\lambda_j)\mathrm{Re}\prod_{k\not=i,j}(\lambda_k+\sqrt{-1}),\]
it is easy to see that
\[
\begin{split}
&(\mathrm{Im}\prod_{k\not=i,j}(\lambda_k+\sqrt{-1}))(\mathrm{Im}\prod_{k=1}^{n}(\lambda_k+\sqrt{-1}))
-(\mathrm{Im}\prod_{k\not=i}(\lambda_k+\sqrt{-1}))
(\mathrm{Im}\prod_{k\not=j}(\lambda_k+\sqrt{-1}))\\
&=-(\mathrm{Im}\prod_{k\not=i,j}(\lambda_k+\sqrt{-1}))^2-(\mathrm{Re}\prod_{k\not=i,j}(\lambda_k+\sqrt{-1}))^2\\
&=-\frac{\prod_{k=1}^{n}(1+\lambda_k^2)}{(1+\lambda_i^2)(1+\lambda_j^2)}
=-\prod_{k=1}^{n}(1+\lambda_k^2)\frac{\sin(\theta_i)}{\sqrt{1+\lambda_i^2}}\frac{\sin(\theta_j)}{\sqrt{1+\lambda_j^2}}.\end{split}\]

Thus,
\[
\begin{split}
\sum_{i,j=1}^{n}\frac{\partial^2 F}{\partial\lambda_i\partial\lambda_j}u_iu_j
&=\frac{-2}{\sin^2(\sum_{k=1}^{n}\theta_k)}(-\cot(\sum_{k=1}^{n}\theta_k)
\sum_{i,j=1}^{n}\frac{u_iu_j}{(1+\lambda_i^2)(1+\lambda_j^2)}+\sum_{i=1}^{n}\frac{\lambda_iu_i^2}{(1+\lambda_i^2)^2})\\
&-\frac{f \prod_{k=1}^{n}(1+\lambda_k^2)}{(\mathrm{Im}\prod_{k=1}^{n}(\lambda_k+\sqrt{-1}))^3} \cdot (\sum_{i,j=1}^{n}\sin(\sum_{k\not=i}\theta_k)\sin(\sum_{k\not=j}\theta_k)
\frac{u_i}{\sqrt{1+\lambda_i^2}}\frac{u_j}{\sqrt{1+\lambda_j^2}}\\
&+\sum_{i=1}^{n}\sin^2(\sum_{k\not=i}\theta_k)\frac{u_i^2}{1+\lambda_i^2}
+\sum_{i=1}^{n}\sum_{j\not=i}\sin(\theta_i)\sin(\theta_j)\frac{u_i}{\sqrt{1+\lambda_i^2}}\frac{u_j}{\sqrt{1+\lambda_j^2}}).
\end{split}\]

Without loss of generality, assume that $\lambda_1\ge\lambda_2...\ge\lambda_n$.
When $n\ge 4$, we first claim that there exists a constant $\epsilon_{10}$ only depending on $n, \Theta_0$ such that
\[\begin{split}
&\frac{-2}{\sin^2(\sum_{k=1}^{n}\theta_k)}(-\cot(\sum_{k=1}^{n}\theta_k)
\sum_{i,j=1}^{n}\frac{u_iu_j}{(1+\lambda_i^2)(1+\lambda_j^2)}+\sum_{i=1}^{n}\frac{\lambda_iu_i^2}{(1+\lambda_i^2)^2})
\\
&\le -\epsilon_{10}\frac{\prod_{k=1}^{n}(1+\lambda_k^2)}{(\mathrm{Im}\prod_{k=1}^{n}(\lambda_k+\sqrt{-1}))^3} \sum_{i=1}^{n}\frac{u_i^2}{1+\lambda_i^2}.
\end{split}\]

The first claim is equivalent to
\[-\cot(\sum_{k=1}^{n}\theta_k)
(\sum_{i=1}^{n}\frac{u_i}{1+\lambda_i^2})^2+\sum_{i=1}^{n}\frac{\lambda_iu_i^2}{(1+\lambda_i^2)^2}\ge \frac{\epsilon_{10}}{2\mathrm{Im}\prod_{k=1}^{n}(\lambda_k+\sqrt{-1})} \sum_{i=1}^{n}\frac{u_i^2}{1+\lambda_i^2}.\]

We divide it into several cases.

In the first case, $\theta_1\le...\le\theta_n\le\frac{\Theta_0}{2}<\frac{\pi}{2}$ and $\cot(\sum_{k=1}^{n}\theta_k)\le \frac{1}{2n}\cot(\frac{\Theta_0}{2})$,
then \[-\cot(\sum_{k=1}^{n}\theta_k)
(\sum_{i=1}^{n}\frac{u_i}{1+\lambda_i^2})^2\ge -\frac{1}{2n}\cot(\frac{\Theta_0}{2})\sum_{i=1}^{n}\frac{\lambda_i u_i^2}{(1+\lambda_i^2)^2}\sum_{i=1}^{n}\frac{1}{\lambda_i}\ge -\frac{1}{2}\sum_{i=1}^{n}\frac{\lambda_i u_i^2}{(1+\lambda_i^2)^2}.\]

Since $\sin(\sum_{k=1}^{n}\theta_k)\ge C_7$,
\[\begin{split}
\frac{1}{2\mathrm{Im}\prod_{k=1}^{n}(\lambda_k+\sqrt{-1})} \sum_{i=1}^{n}\frac{u_i^2}{1+\lambda_i^2}\le
\frac{1}{2C_7 \prod_{k=1}^{n}\sqrt{1+\lambda_k^2}} \sum_{i=1}^{n}\frac{u_i^2}{1+\lambda_i^2}\\
\le\frac{1}{2C_7}\sum_{i=1}^{n}\frac{\lambda_i u_i^2}{\lambda_i^2(1+\lambda_i^2)}
\le\frac{\cot^2(\frac{\Theta_0}{2})+1}{2C_7\cot^2(\frac{\Theta_0}{2})}\sum_{i=1}^{n}\frac{\lambda_i u_i^2}{(1+\lambda_i^2)^2}.
\end{split}\]

So we get the required estimate if $\epsilon_{10}<\frac{C_7\cot^2(\frac{\Theta_0}{2})}{\cot^2(\frac{\Theta_0}{2})+1}$.

In the second case, $\theta_1\le...\le\theta_n\le\frac{\Theta_0}{2}<\frac{\pi}{2}$ and
$\cot(\sum_{k=1}^{n}\theta_k)>\frac{1}{2n}\cot(\frac{\Theta_0}{2})$, then
\[\sum_{k=1}^{n}\theta_k<\arccot(\frac{1}{2n}\cot(\frac{\Theta_0}{2}))<\frac{\pi}{2}.\]

So \[-\cot(\sum_{k=1}^{n}\theta_k)
(\sum_{i=1}^{n}\frac{u_i}{1+\lambda_i^2})^2\ge -\cot(\sum_{k=1}^{n}\theta_k)\sum_{i=1}^{n}\frac{\lambda_i u_i^2}{(1+\lambda_i^2)^2}\sum_{i=1}^{n}\frac{1}{\lambda_i}.\]

If $\alpha,\beta>0$ and $\alpha+\beta<\frac{\pi}{2}$, then
\[\tan(\alpha+\beta)=\frac{\tan(\alpha)+\tan(\beta)}{1-\tan(\alpha)\tan(\beta)}\ge\tan(\alpha)+\tan(\beta).\]
So
\[\begin{split}
&1-\cot(\sum_{k=1}^{n}\theta_k)\sum_{i=1}^{n}\frac{1}{\lambda_i}= 1-\cot(\sum_{k=1}^{n}\theta_k)\sum_{i=1}^{n}\tan(\theta_i)\\
&\ge 1-\cot(\sum_{k=1}^{n}\theta_k)(\tan(\theta_n)+\tan(\sum_{i=1}^{n-1}\theta_i))
=\tan(\theta_n)\tan(\sum_{i=1}^{n-1}\theta_i)\\
&\ge \tan(\theta_n)\sum_{i=1}^{n-1}\tan(\theta_i)=\frac{1}{\lambda_n}\sum_{i=1}^{n-1}\frac{1}{\lambda_i}\ge \frac{1}{\lambda_n\lambda_{n-1}}.
\end{split}\]

As in (1), we know that
\[\frac{\mathrm{Im}\prod_{k=1}^{n}(\lambda_k+\sqrt{-1})}{\prod_{k=1}^{n}\sqrt{1+\lambda_k^2}}
=\sin(\sum_{k=1}^{n}\theta_k)\ge C_8\sum_{k=1}^{n}\theta_k\ge \frac{C_8}{C_9}\sum_{k=1}^{n}\frac{1}{\lambda_k}\ge \frac{C_8}{C_9\lambda_n}.\]
So
\[\mathrm{Im}\prod_{k=1}^{n}(\lambda_k+\sqrt{-1})\ge \frac{C_8}{C_9\lambda_n}\prod_{k=1}^{n}\sqrt{1+\lambda_k^2}
\ge \frac{C_8}{C_9}\prod_{k=1}^{n-1}\sqrt{1+\lambda_k^2}\ge\frac{C_8}{C_9}\lambda_1\lambda_{n-2}\lambda_{n-1}
\ge\frac{C_8}{C_9}\lambda_1\lambda_{n-1}\lambda_{n}.\]

So \[\begin{split}
\frac{1}{2\mathrm{Im}\prod_{k=1}^{n}(\lambda_k+\sqrt{-1})} \sum_{i=1}^{n}\frac{u_i^2}{1+\lambda_i^2}\le
\frac{C_9}{2C_8}\frac{1}{\lambda_n\lambda_{n-1}} \sum_{i=1}^{n}\frac{\lambda_i u_i^2}{\lambda_i^2(1+\lambda_i^2)}\\
\le\frac{C_9(\cot^2(\frac{\Theta_0}{2})+1)}{2C_8\cot^2(\frac{\Theta_0}{2})\lambda_{n}\lambda_{n-1}}
\sum_{i=1}^{n}\frac{\lambda_i u_i^2}{(1+\lambda_i^2)^2}.
\end{split}\]

So we get the required estimate if $\frac{C_9\epsilon_{10}(\cot^2(\frac{\Theta_0}{2})+1)}{2C_8\cot^2(\frac{\Theta_0}{2})}\le 1$.

In the third case, $\theta_n>\frac{\Theta_0}{2}$. So $\sum_{k=1}^{n-1}\theta_k<\frac{\Theta_0}{2}<\frac{\pi}{2}$.
Similar to the second case,

\[-\cot(\sum_{k=1}^{n-1}\theta_k)
(\sum_{i=1}^{n-1}\frac{u_i}{1+\lambda_i^2})^2+\sum_{i=1}^{n-1}\frac{\lambda_iu_i^2}{(1+\lambda_i^2)^2}\ge \frac{1}{\lambda_{n-1}\lambda_{n-2}}\sum_{i=1}^{n-1}\frac{\lambda_iu_i^2}{(1+\lambda_i^2)^2}.\]

We already know that $\sin(\sum_{k=1}^{n}\theta_k)\ge C_7$,
so \[\begin{split}
&\frac{1}{2\mathrm{Im}\prod_{k=1}^{n}(\lambda_k+\sqrt{-1})} \sum_{i=1}^{n-1}\frac{u_i^2}{1+\lambda_i^2}\le
\frac{1}{2C_7\prod_{k=1}^{n}\sqrt{1+\lambda_k^2}} \sum_{i=1}^{n-1}\frac{u_i^2}{1+\lambda_i^2}\\
&\le\frac{1}{2C_7\lambda_{n-1}\lambda_{n-2}\lambda_1} \sum_{i=1}^{n-1}\frac{u_i^2}{1+\lambda_i^2}
\le\frac{1}{2C_7\lambda_{n-1}\lambda_{n-2}} \sum_{i=1}^{n-1}\frac{\lambda_i u_i^2}{\lambda_i^2(1+\lambda_i^2)}\\
&\le\frac{\cot^2(\frac{\Theta_0}{2})+1}{2C_7\cot^2(\frac{\Theta_0}{2})\lambda_{n-1}\lambda_{n-2}}
\sum_{i=1}^{n-1}\frac{\lambda_i u_i^2}{(1+\lambda_i^2)^2}.
\end{split}\]

On the other hands,

\[\begin{split}
&-\cot(\sum_{k=1}^{n}\theta_k)
(\sum_{i=1}^{n}\frac{u_i}{1+\lambda_i^2})^2+\cot(\sum_{k=1}^{n-1}\theta_k)
(\sum_{i=1}^{n-1}\frac{u_i}{1+\lambda_i^2})^2+\frac{\lambda_n u_n^2}{(1+\lambda_n^2)^2}\\
&=-2\cot(\sum_{k=1}^{n}\theta_k)
(\sum_{i=1}^{n-1}\frac{u_i}{1+\lambda_i^2})\frac{u_n}{1+\lambda_n^2}
+(\cot(\sum_{k=1}^{n-1}\theta_k)-\cot(\sum_{k=1}^{n}\theta_k))
(\sum_{i=1}^{n-1}\frac{u_i}{1+\lambda_i^2})^2\\
&+\frac{(\lambda_n-\cot(\sum_{k=1}^{n}\theta_k)) u_n^2}{(1+\lambda_n^2)^2}\\
&\ge (\lambda_n-\cot(\sum_{k=1}^{n}\theta_k)
-\frac{\cot^2(\sum_{k=1}^{n}\theta_k)}{\cot(\sum_{k=1}^{n-1}\theta_k)-\cot(\sum_{k=1}^{n}\theta_k)})
\frac{u_n^2}{(1+\lambda_n^2)^2}\\
&=\frac{\lambda_n+\cot(\sum_{k=1}^{n-1}\theta_k)}{\cot^2(\sum_{k=1}^{n-1}\theta_k)+1}
\frac{u_n^2}{(1+\lambda_n^2)^2}=\sin^2(\sum_{k=1}^{n-1}\theta_k)(\frac{\cos(\theta_n)}{\sin(\theta_n)}
+\cot(\sum_{k=1}^{n-1}\theta_k))\frac{u_n^2}{(1+\lambda_n^2)^2}\\
&=\frac{\sin(\sum_{k=1}^{n-1}\theta_k)\sin(\sum_{k=1}^{n}\theta_k)}{\sin(\theta_n)}\frac{u_n^2}{(1+\lambda_n^2)^2}
\ge\frac{C_8}{C_9\lambda_{n-1}}C_7\frac{u_n^2}{1+\lambda_n^2}\frac{1}{\sqrt{1+\lambda_n^2}}.
\end{split}\]

It is easy to see that

\[\frac{1}{2\mathrm{Im}\prod_{k=1}^{n}(\lambda_k+\sqrt{-1})} \frac{u_n^2}{1+\lambda_n^2}\le
\frac{1}{2C_7\prod_{k=1}^{n}\sqrt{1+\lambda_k^2}}\frac{u_n^2}{1+\lambda_n^2}
\le\frac{1}{2C_7\lambda_{n-1}}\frac{u_n^2}{1+\lambda_n^2}\frac{1}{\sqrt{1+\lambda_n^2}}.\]

Therefore, as long as $\frac{\epsilon_{10}(\cot^2(\frac{\Theta_0}{2})+1)}{2C_7\cot^2(\frac{\Theta_0}{2})}<1$ and
$\frac{\epsilon_{10}}{2C_7}<\frac{C_8}{C_9}C_7$, we get the required estimate.

Thus, when $n\ge 4$, we have proved that
\[\begin{split}
&\frac{-2}{\sin^2(\sum_{k=1}^{n}\theta_k)}(-\cot(\sum_{k=1}^{n}\theta_k)
\sum_{i,j=1}^{n}\frac{u_iu_j}{(1+\lambda_i^2)(1+\lambda_j^2)}+\sum_{i=1}^{n}\frac{\lambda_iu_i^2}{(1+\lambda_i^2)^2})
\\
&\le -\epsilon_{10}\frac{\prod_{k=1}^{n}(1+\lambda_k^2)}{(\mathrm{Im}\prod_{k=1}^{n}(\lambda_k+\sqrt{-1}))^3} \sum_{i=1}^{n}\frac{u_i^2}{1+\lambda_i^2}
\end{split}\]

When $n=2,3$, a similar argument implies that
\[\frac{-2}{\sin^2(\sum_{k=1}^{n}\theta_k)}(-\cot(\sum_{k=1}^{n}\theta_k)
\sum_{i,j=1}^{n}\frac{u_iu_j}{(1+\lambda_i^2)(1+\lambda_j^2)}+\sum_{i=1}^{n}\frac{\lambda_iu_i^2}{(1+\lambda_i^2)^2})\le 0.\]

Compared to Theorem 1.1 of \cite{Takahashi}, the main improvement is that we choose the variable $\hat\Theta$ in Theorem 1.1 of \cite{Takahashi} as $(n-1)\frac{\pi}{2}$ and we also have a better estimate $-\epsilon_{10}\frac{\prod_{k=1}^{n}(1+\lambda_k^2)}{(\mathrm{Im}\prod_{k=1}^{n}(\lambda_k+\sqrt{-1}))^3} \sum_{i=1}^{n}\frac{u_i^2}{1+\lambda_i^2}$ which will be used to deal with terms involving $f$ when $n\ge 4$.

The next goal is to prove that when $n\ge 3$, the matrix corresponding to
\[\sum_{i=1}^{n}\sin^2(\sum_{k\not=i}\theta_k)v_i^2
+\sum_{i=1}^{n}\sum_{j\not=i}\sin(\theta_i)\sin(\theta_j)v_iv_j\]
is positive definite. When $\theta_1=\theta_2=...=\theta_n$ and it is sufficiently small, it is easy to see that $\sin^2(\sum_{k\not=i}\theta_k)>\sin^2(\theta_i)$. So the matrix is indeed positive definite. Since the space $\bar\Gamma_{\theta_0,\Theta_0}$ is path connected, it suffices to show that the determinant of the matrix is positive on $\bar\Gamma_{\theta_0,\Theta_0}$.

Without loss of generality, assume that $\theta_n\ge\theta_{n-1}...\ge\theta_1$. When $i\not=n$, \[\theta_i\le\theta_n<\sum_{k\not=i}\theta_k\le\Theta_0-\theta_i<\pi-\theta_i,\] so
$\sin(\theta_i)<\sin(\sum_{k\not=i}\theta_k)$. When $\sin^2(\sum_{k=1}^{n-1}\theta_k)\not=\sin^2(\theta_n)$, let $A$ be the complex diagonal matrix such that
\[A_{ii}=\sqrt{\sin^2(\sum_{k\not=i}\theta_k)-\sin^2(\theta_i)},\] defined \[B=(\frac{\sin(\theta_1)}{A_{11}},...,\frac{\sin(\theta_n)}{A_{nn}}).\]
Then we need to compute $\det A^T(I+B^TB)A$.
By elementary linear algebra,
\[\begin{split}
\det A^T(I+B^TB)A&=(\det A)^2(1+BB^T)\\
&=\prod_{i=1}^{n}(\sin^2(\sum_{k\not=i}\theta_k)-\sin^2(\theta_i))
(1+\sum_{i=1}^{n}\frac{\sin^2(\theta_i)}{\sin^2(\sum_{k\not=i}\theta_k)-\sin^2(\theta_i)})\\
&=\prod_{i=1}^{n}(\sin^2(\sum_{k\not=i}\theta_k)-\sin^2(\theta_i))
+\sum_{i=1}^{n}\sin^2(\theta_i)\prod_{j\not=i}(\sin^2(\sum_{k\not=j}\theta_k)-\sin^2(\theta_j)).
\end{split}\]
By continuity, this equation also holds when $\sin^2(\sum_{k=1}^{n-1}\theta_k)=\sin^2(\theta_n)$.

Therefore, when $\sin^2(\sum_{k=1}^{n-1}\theta_k)\ge\sin^2(\theta_n)$, we already get the required inequality. We only need to prove that \[\sum_{i=1}^{n}\frac{\sin^2(\theta_i)}{\sin^2(\sum_{k\not=i}\theta_k)-\sin^2(\theta_i)}<-1\] when $\sin^2(\sum_{k=1}^{n-1}\theta_k)<\sin^2(\theta_n)$. In this case, $\sum_{k=1}^{n-1}\theta_k<\theta_n$.

Now we want to study the function
\[G(\alpha,\beta)=\frac{\sin^2(\beta)}{\sin^2(\alpha-\beta)-\sin^2(\beta)}
=\frac{-2\sin^2(\beta)}{\cos(2\alpha-2\beta)-\cos(2\beta)}
=\frac{\sin^2(\beta)}{\sin(\alpha)\sin(\alpha-2\beta)}\]
for any $0<\beta<\frac{\alpha}{2}<\frac{\pi}{2}$.
Then

\[\frac{\partial G}{\partial \beta}
=\frac{2\sin(\beta)\cos(\beta)\sin(\alpha-2\beta)+2\sin^2(\beta)\cos(\alpha-2\beta)}{\sin(\alpha)\sin^2(\alpha-2\beta)}
=\frac{2\sin(\beta)\sin(\alpha-\beta)}{\sin(\alpha)\sin^2(\alpha-2\beta)}>0,\]
so
\[\begin{split}
\frac{\partial}{\partial \beta} \log (\frac{\partial G}{\partial \beta})
&=\cot(\beta)-\cot(\alpha-\beta)+4\cot(\alpha-2\beta)\\
&=\cot(\beta)-\cot(\alpha-\beta)+4\frac{\cot(\beta)\cot(\alpha-\beta)+1}{\cot(\beta)-\cot(\alpha-\beta)}\\
&=\frac{(\cot(\beta)+\cot(\alpha-\beta))^2+4}{\cot(\beta)-\cot(\alpha-\beta)}.
\end{split}\]

Since $0<\beta<\alpha-\beta<\pi$, we know that $\cot(\beta)-\cot(\alpha-\beta)>0$, so $\frac{\partial^2 G}{\partial \beta^2}>0$.
Therefore, when we replace $\theta_1$ by $0$ and replace $\theta_{n-1}$ by $\theta_{n-1}+\theta_1$, we see that $\sum_{i=1}^{n}\frac{\sin^2(\theta_i)}{\sin^2(\sum_{k\not=i}\theta_k)-\sin^2(\theta_i)}$ strictly increase. We can repeat the process to prove that
\[\sum_{i=1}^{n}\frac{\sin^2(\theta_i)}{\sin^2(\sum_{k\not=i}\theta_k)-\sin^2(\theta_i)}
<\frac{\sin^2(\theta_n)}{\sin^2(\sum_{k=1}^{n-1}\theta_k)-\sin^2(\theta_n)}
+\frac{\sin^2(\sum_{k=1}^{n-1}\theta_k)}{\sin^2(\theta_n)-\sin^2(\sum_{k=1}^{n-1}\theta_k)}=-1.\]
This is the required inequality.

Thus, when $n\ge 3$, we have proved that
\[\sum_{i,j=1}^{n}\frac{\partial^2}{\partial\lambda_i\partial\lambda_j}
(\frac{1}{\mathrm{Im}\prod_{k=1}^{n}(\lambda_k+\sqrt{-1})})u_iu_j\ge 0.\]
When $n=2$, it is also true because
\[\sum_{i=1}^{n}\sin^2(\sum_{k\not=i}\theta_k)v_i^2
+\sum_{i=1}^{n}\sum_{j\not=i}\sin(\theta_i)\sin(\theta_j)v_iv_j\ge 0\] by Cauchy-Schwarz inequality.

Therefore, when $f\ge 0$, we get (4) as long as $\epsilon_5<\epsilon_{10}$. When $-\epsilon_2\le f<0$ and $n\ge 4$, using the bound that $|\sin(\theta_i)|\le 1$ and $|\sin(\sum_{k\not=i}\theta_k)|\le 1$, it is easy to see that

\[
\begin{split}
&-\frac{f \prod_{k=1}^{n}(1+\lambda_k^2)}{(\mathrm{Im}\prod_{k=1}^{n}(\lambda_k+\sqrt{-1}))^3} \cdot (\sum_{i,j=1}^{n}\sin(\sum_{k\not=i}\theta_k)\sin(\sum_{k\not=j}\theta_k)
\frac{u_i}{\sqrt{1+\lambda_i^2}}\frac{u_j}{\sqrt{1+\lambda_j^2}}\\
&+\sum_{i=1}^{n}\sin^2(\sum_{k\not=i}\theta_k)\frac{u_i^2}{1+\lambda_i^2}
+\sum_{i=1}^{n}\sum_{j\not=i}\sin(\theta_i)\sin(\theta_j)\frac{u_i}{\sqrt{1+\lambda_i^2}}\frac{u_j}{\sqrt{1+\lambda_j^2}})\\
&\le \frac{3n\epsilon_2\prod_{k=1}^{n}(1+\lambda_k^2)}{(\mathrm{Im}\prod_{k=1}^{n}(\lambda_k+\sqrt{-1}))^3}
\sum_{i=1}^{n}\frac{u_i^2}{1+\lambda_i^2}.
\end{split}\]
Therefore, as long as $\epsilon_2<\frac{\epsilon_{10}}{6n}$ and $\epsilon_5<\frac{\epsilon_{10}}{2}$, we get the required estimate.

(5) Suppose that $\lambda\in\bar\Gamma_{\theta_0,\Theta_0}$, and $F(\lambda)=0$. For any $i=1,2,3,...n$, using (3), we see that $F(\lambda)$ is strictly smaller than the limit of $F$ when we fix $\lambda_k$ for $k\not=i$ and let $\lambda_i$ go to infinity. Using a similar argument as (1), we see that
\[\mathrm{Im}\prod_{k=1}^{n}(\lambda_k+\sqrt{-1})\ge \min\{C_7,\frac{(n-1)C_8}{C_9}\cot^{n-2}(\frac{\Theta_0}{2})\}\lambda_i.\]
So the limit of $F$ is $\cot(\sum_{k\not=i}\arccot(\lambda_k))-\cot(\theta_0)$.
So $\sum_{k\not=i}\arccot(\lambda_k)<\theta_0$.

Moreover, using the fact that
\[0=\cot(\sum_{k=1}^{n}\arccot(\lambda_k))-\frac{f}{\mathrm{Im}\prod_{k=1}^{n}(\lambda_k+\sqrt{-1})}-\cot(\theta_0)
\le\cot(\sum_{k=1}^{n}\arccot(\lambda_k))-\cot(\theta_0)+\frac{\epsilon_2}{C_6},\]
we see that $\sum_{k=1}^{n}\arccot(\lambda_k)<\Theta_0$ as long as $\epsilon_2<C_6(\cot(\theta_0)-\cot(\Theta_0))$.

Thus, $\lambda\in \Gamma_{\theta_0,\Theta_0}$.

(6) Let $\lambda$ be any element in $\Gamma_{\theta_0,\Theta_0}$. Let $\lambda'$ be any element in $\Gamma_{\theta_0,\Theta_0}$ such that $F(\lambda')=0$ and $\lambda'_k\ge\lambda_k$ for all $k=1,2,3,...,n$.

Then for any $i=1,2,3,...n$, $\sum_{k\not=i}\arccot(\lambda_k)<\theta_0$. If \[\lambda'_i>\cot(\frac{\theta_0-\sum_{k\not=i}\arccot(\lambda_k)}{2})\]
and \[\min\{C_7,\frac{(n-1)C_8}{C_9}\cot^{n-2}(\frac{\Theta_0}{2})\}\lambda'_i >\frac{|f|}{\cot(\frac{\theta_0+\sum_{k\not=i}\arccot(\lambda_k)}{2})-\cot(\theta_0)},\]
we get a direct contradiction to the estimate that
\[0\ge\cot(\sum_{k=1}^{n}\arccot(\lambda'_k))-\cot(\theta_0) -\frac{|f|}{\min\{C_7,\frac{(n-1)C_8}{C_9}\cot^{n-2}(\frac{\Theta_0}{2})\}\lambda'_i}.\]

(7) Fix any $\lambda\in\bar\Gamma_{\theta_0,\Theta_0}$. Consider the set $C$ of $\lambda'\in\bar\Gamma_{\theta_0,\Theta_0}$ such that $t\lambda+(1-t)\lambda'\in \bar\Gamma_{\theta_0,\Theta_0}$
for all $t\in[0,1]$. It is easy to see that $C$ is a closed set in the relative topology on $\bar\Gamma_{\theta_0,\Theta_0}$. Now if $\lambda'$ is in this set. For any $\lambda''\in \bar\Gamma_{\theta_0,\Theta_0}$ sufficiently close to $\lambda'$, there exist $\pi>\Theta'_0>\Theta_0>\theta'_0>\theta_0$ such that $t\lambda+(1-t)\lambda'\in \bar\Gamma_{\theta'_0,\Theta'_0}$. By (4) applied to the set $\bar\Gamma_{\theta'_0,\Theta'_0}$ and the case $f=0$,
we see that $\cot(\sum_{k=1}^{n}\arccot(t\lambda_k+(1-t)\lambda'_k))$ is a concave function. So it is at least $\Theta_0$. The similar arguments can be applied to $\cot(\sum_{k\not=i}\arccot(t\lambda_k+(1-t)\lambda'_k))$ for any $i=1,2,3,...,n$. So we see that $\lambda''$ is in $C$. In other words, $C$ is also open in the relative topology. Since $\bar\Gamma_{\theta_0,\Theta_0}$ is connected, and $\lambda\in C$, we see that $\bar\Gamma_{\theta_0,\Theta_0}=C$. So $\bar\Gamma_{\theta_0,\Theta_0}$ is convex because for any $\lambda,\lambda'\in\bar\Gamma_{\theta_0,\Theta_0}$, $t\lambda+(1-t)\lambda'\in \bar\Gamma_{\theta_0,\Theta_0}$.

(8) This follows from the concaveness of \[\begin{split}
F(\lambda_1,...\lambda_{i-1},t\lambda_i+(1-t)\lambda_j,\lambda_{i+1},..., \lambda_{j-1},t\lambda_j+(1-t)\lambda_i,\lambda_{j+1},...,\lambda_{n})\\
=F(\lambda_1,...\lambda_{i-1},t\lambda_j+(1-t)\lambda_i,\lambda_{i+1},..., \lambda_{j-1},t\lambda_i+(1-t)\lambda_j,\lambda_{j+1},...,\lambda_{n})
\end{split}\] for $t\in[0,1]$.

(9) This follows from (4), (8) and the result in \cite{Spruck} which was also used as Equation (66) in \cite{Szekelyhidi}.

(10) It is similar to (7).

\end{proof}

As a corollary, we get the following \textit{a priori} estimate:

\begin{corollary}
Let $M^n$ be a K\"ahler manifold with a K\"ahler metric $\chi$ and a real closed (1,1)-form $\omega_0$. Let $\theta_0\in(0,\pi)$ be a constant and let $\Theta_0\in(\theta_0,\pi)$ be another constant. Then there exists a constant $\epsilon_2>0$ only depending on $n, \theta_0, \Theta_0$ such that the following holds:

Assume that (1) When $n\ge 4$, $f>-\epsilon_2$ is a smooth function;

(2) When $n=1,2,3$, $f\ge 0$ is a constant;

(3) $\omega_0\in \Gamma_{\chi,\theta_0,\Theta_0}$.

Assume that $\varphi$ is a smooth function satisfying
$\sup_M\varphi=0$, $\omega_\varphi=\omega_0+\sqrt{-1}\partial\bar\partial\varphi \in \Gamma_{\chi,\theta_0,\Theta_0}$, and
\[\mathrm{Re}(\omega_\varphi+\sqrt{-1}\chi)^n-\cot(\theta_0)\mathrm{Im}(\omega_\varphi+\sqrt{-1}\chi)^n-f\chi^n=0.\]
Then for any $k\in\mathbb{N}$, any $\alpha\in(0,1)$, there exists a constant $C_{11}$ only depending on $M$, $n$, $\chi$, $||\omega_0||_{C^{\infty}(\chi)}$, $\theta_0$, $\Theta_0$, $||f||_{C^{\infty}(\chi)}$, $k$, $\alpha$, $\max_{x\in M} (\theta_0-P_\chi(\omega_0)(x))$, $\max_{x\in M} (\Theta_0-Q_\chi(\omega_0)(x))$ such that \[||\varphi||_{C^{k,\alpha}(\chi)}\le C_{11}.\]
\label{A-priori-estimate}
\end{corollary}

\begin{proof}
The $C^0$ estimate is similar to Proposition 10 of \cite{Szekelyhidi}. Then we get the relationship between the $C^2$ bound and the $C^1$ bound similar to Theorem 4.1 of \cite{Szekelyhidi} with the adjustments given in page 34-35 of \cite{Chen}. Then we get the $C^2$ estimate by Proposition 5.1 of \cite{CollinsJacobYau}. The $C^{2,\alpha}$ estimate is given by Evans-Krylov theory \cite{Evans, Krylov}. The higher order estimates are given by standard Schauder estimates. In all of above, the key properties used are all provided in Lemma \ref{Property-of-F}.
\end{proof}

Then we can use the continuity method to prove Proposition \ref{Solvability-assuming-sub-solution}. We first choose the path
\[\omega_{1,s}=s (\omega_0-\cot(\theta_0)\chi)+\cot(\frac{\theta_0}{2n})\chi\] and
$f_{1,s}$ as the constant satisfying \[\int_M f_{1,s}\chi^n=\int_M(\mathrm{Re}(\omega_{1,s}+\sqrt{-1}\chi)^n-\cot(\theta_0)\mathrm{Im}(\omega_{1,s}+\sqrt{-1}\chi)^n)\] for $s\in[0,1]$.
Then since $\omega_{1,s}\ge\cot(\frac{\theta_0}{2n})\chi$, we see that $P_\chi(\omega_{1,s})\le Q_\chi(\omega_{1,s})\le \frac{\theta_0}{2}$ and $f_{1,s}\ge0$. Let $I_1$ be the set of $s$ such that there exists a smooth function $\varphi_s$ satisfying $\omega_{1,\varphi_s,s}=\omega_{1,s}+\sqrt{-1}\partial\bar\partial\varphi_s\in \Gamma_{\chi,\theta_0,\Theta_0}$ and
\[\mathrm{Re}(\omega_{1,\varphi_s,s}+\sqrt{-1}\chi)^n-\cot(\theta_0)\mathrm{Im}(\omega_{1,\varphi_s,s}+\sqrt{-1}\chi)^n-f_{1,s}\chi^n=0.\]
Then $0\in I_1$. The openness of $I_1$ follows the integrability condition, implicit function theorem and standard elliptic estimates. The closedness of $I_1$ follows from Corollary \ref{A-priori-estimate} and Lemma \ref{Property-of-F} (5). So $1\in I_1$.

Then we choose the path \[\omega_{2,s}=\omega_0+s\chi\] and
$f_{2,s}$ as the constant satisfying \[\int_M f_{2,s}\chi^n=\int_M(\mathrm{Re}(\omega_{2,s}+\sqrt{-1}\chi)^n-\cot(\theta_0)\mathrm{Im}(\omega_{2,s}+\sqrt{-1}\chi)^n)\] for $s\in[0,\cot(\frac{\theta_0}{2n})-\cot(\theta_0)]$. Then we see that $P_\chi(\omega_{2,s})\le P_\chi(\omega_0)$ and $Q_\chi(\omega_{2,s})\le Q_\chi(\omega_0)$. Moreover,
\[\frac{\partial f_{2,s}}{\partial s} \int_M \chi^n=\int_M n(\mathrm{Re}(\omega_{2,s}+\sqrt{-1}\chi)^{n-1}-\cot(\theta_0)\mathrm{Im}(\omega_{2,s}+\sqrt{-1}\chi)^{n-1})\wedge \chi\ge0\]
by Lemma 8.2 of \cite{CollinsJacobYau}. So $f_{2,s}\ge 0$. Finally, we fix $\omega_0$ and choose the path that
\[f_{3,s}=sf+(1-s)\frac{\int_M f\chi^n}{\int_M \chi^n}\] for $s\in [0,1]$.
We omit the arguments for the second path and the third path because they are similar to the first path. This finishes the proof of Proposition \ref{Solvability-assuming-sub-solution}.

\section{Current as a subsolution}

In this section, we prove Proposition \ref{Current-sub-solution}.

First of all, we make sense of the statement $\omega\in \bar\Gamma_{\chi,\theta_0,\Theta_0}$, when $\omega$ is the sum of a real closed (1,1)-form and a closed positive (1,1)-current.

The definition of smoothing is the same as Definition 3.1 of \cite{Chen}:
\begin{definition}
Fix a smooth non-negative function $\rho$ supported in [0,1] such that
\[\int_0^1\rho(t)t^{2n-1}\mathrm{Vol}(\partial B_1(0))dt=1.\]
For any $\delta>0$, the smoothing $\varphi_\delta$ is defined by
\[\varphi_\delta(x)=\int_{\mathbb{C}^n} \varphi(x-y)\delta^{-2n}\rho(|\frac{y}{\delta}|)d\mathrm{Vol}_y.\]
We can define the smoothing of a current using the similar formula. It is easy to see that the smoothing commutes with derivatives. So $(\sqrt{-1}\partial\bar\partial\varphi)_\delta=\sqrt{-1}\partial\bar\partial(\varphi_\delta)$.
\label{Definition-smoothing}
\end{definition}

Then we define $\omega\in \bar\Gamma_{\chi_0,\theta_0,\Theta_0}$ when $\chi_0$ is a K\"ahler form with constant coefficients on an open set $O\subset\mathbb{C}^n$:

\begin{definition}
Suppose that $\chi_0$ is a K\"ahler form with constant coefficients on an open set $O\subset\mathbb{C}^n$ and $\omega$ is the sum of a real closed (1,1)-form and a closed positive (1,1)-current. Then we say that $\omega\in \bar\Gamma_{\chi_0,\theta_0,\Theta_0}$ on $O$ if for any $\delta>0$, the smoothing $\omega_\delta$ satisfies $\omega_\delta\in \bar\Gamma_{\chi_0,\theta_0,\Theta_0}$ on the set $O_\delta=\{x:B_\delta(x)\subset O\}$.
\end{definition}

For general $\chi$, we make the following definition:

\begin{definition}
We say that $\omega\in \bar\Gamma_{\chi,\theta_0,\Theta_0}$ if for any $\epsilon_{12}>0$ and $\epsilon_{13}>0$ satisfying \[(1+\epsilon_{13})(\cot(\theta_0)+\epsilon_{12})>\cot(\theta_0),\] on any coordinate chart, for any open subset $O$, if $(1+\epsilon_{13})\chi_0\ge\chi\ge\chi_0$ on $O$ for a K\"ahler form $\chi_0$ with constant coefficients, then $\omega+\epsilon_{12}\chi \in \bar\Gamma_{\chi_0,\theta_0,\Theta_0}$.
\label{Definition-current-sub-solution}
\end{definition}

\begin{remark}
By Lemma \ref{Property-of-F} (10), when $\omega$ is smooth, it is easy to see that $\omega\in \bar\Gamma_{\chi,\theta_0,\Theta_0}$ in the sense of Definition \ref{Definition-current-sub-solution} if and only if $\omega\in \bar\Gamma_{\chi,\theta_0,\Theta_0}$ in the usual sense. Another key property is that the condition $\omega\in \bar\Gamma_{\chi,\theta_0,\Theta_0}$ in the sense of Definition \ref{Definition-current-sub-solution} is preserved under weak limit in the sense of current.
\label{Property-subsolution}
\end{remark}

Then we need the following lemma which is similar to Lemma 5.5 of \cite{Chen}:

\begin{lemma}
Suppose that $A$ is a $p\times p$ Hermitian matrix, $B$ is a diagonal $q\times q$ Hermitian matrix with $B_{ii}=\lambda_i$, $C$ is a $p\times q$ complex matrix, $D$ is another diagonal $q\times q$ matrix such that
\[D_{ii}=\frac{\mathrm{Im}\prod_{k\not=i}(\lambda_k+\sqrt{-1})}{\mathrm{Im}\prod_{k=1}^{q}(\lambda_k+\sqrt{-1})}.\]
Suppose that
\[Q_I(\begin{bmatrix}
    A       & C \\
    \bar C^{T} & B
\end{bmatrix})<\pi.
\]
Then $D$ is well defined. Moreover,
\[Q_{I}(A-CD\bar C^{T})+Q_I(B)\le Q_I(\begin{bmatrix}
    A       & C \\
    \bar C^{T} & B
\end{bmatrix})\]
and
\[P_{I}(A-CD\bar C^{T})+Q_I(B)\le P_I(\begin{bmatrix}
    A       & C \\
    \bar C^{T} & B
\end{bmatrix}).\]
\label{Subadditive-submatrix}
\end{lemma}

\begin{proof}
Define $E$, $F$ be diagonal $q\times q$ matrices such that
\[E_{ii}=\frac{\lambda_i}{1+\lambda_i^2}, F_{ii}=\frac{1}{1+\lambda_i^2}.\]
We first claim that \[Q_{I+CF\bar C^{T}}(A-CE\bar C^{T})+Q_I(B)=Q_I(\begin{bmatrix}
    A       & C \\
    \bar C^{T} & B
\end{bmatrix}).\]
In fact, it is easy to see that
\[\begin{split}\det\begin{bmatrix}
    A+\sqrt{-1}I   & C \\
    \bar C^{T} & B+\sqrt{-1}I
\end{bmatrix}=\det\begin{bmatrix}
    A-C(B+\sqrt{-1}I)^{-1}\bar C^{T}+\sqrt{-1}I              & O \\
    O & B+\sqrt{-1}I
\end{bmatrix}.
\end{split}\]
Since \[A-C(B+\sqrt{-1}I)^{-1}\bar C^{T}+\sqrt{-1}I=A-CE\bar C^{T}+ \sqrt{-1}(I+CF\bar C^{T}),\]
it is easy to see that

\[Q_{I+CF\bar C^{T}}(A-CE\bar C^{T})+Q_I(B) \equiv Q_I(\begin{bmatrix}
    A       & C \\
    \bar C^{T} & B
\end{bmatrix}) \mod 2\pi.\]
This is also true when we replace $A$ by $A + t I$ and replace $B$ by $B + t I$ for $t\ge 0$. However, when $t$ is large enough, all the quantities are close to 0. So there is no multiple of $2\pi$ there. By continuity, there is also no multiple of $2\pi$ when $t=0$.

As a corollary of the claim, we see that $Q_I(B)<\pi$. This implies that
\[\mathrm{Im}\prod_{k=1}^{q}(\lambda_k+\sqrt{-1})>0.\]
So $D$ is well-defined.

Moreover,
\[D_{ii}-E_{ii} =\frac{-\mathrm{Re}\prod_{k=1}^{q}(\lambda_k+\sqrt{-1})}{(1+\lambda_i^2)(\mathrm{Im}\prod_{k=1}^{q}(\lambda_k+\sqrt{-1}))} =-\cot(Q_I(B))F_{ii}.\]

Now we write $A-CE\bar C^{T}$ as $a_{ij}$, and write $CF\bar C^{T}$ as $b_{ij}$. Define
\[a=\sqrt{-1}\sum_{i,j=1}^{q}a_{ij}dz^i\wedge d\bar z^j, b=\sqrt{-1}\sum_{i,j=1}^{q}b_{ij}dz^i\wedge d\bar z^j, c=\sqrt{-1}\sum_{i=1}^{q}dz^i\wedge d\bar z^i.\]
Then $Q_{c+b}(a)<\pi-Q_I(B)<\pi$.
Now we define \[I=\{t\in[0,1]\ \text {such that}\ Q_{c_s}(a_s)\le Q_{c+b}(a)\ \text{for all}\ s\in[0,t]\},\]
where $a_s=a+s\cot(Q_I(B))b$ and $c_s=c+b-sb$.
If $b=0$, then it is trivial that $I=[0,1]$. So we only need to consider the case when $b\not=0$.
It is also trivial that $0\in I$ and $I$ is closed. Now we assume that $t\in I$. Then
\[\begin{split}
&\frac{d}{ds}|_{s=t} \cot(Q_{c_s}(a_s))=\frac{d}{ds}|_{s=t}\frac{\mathrm{Re}(a_s+\sqrt{-1}c_s)^n}{\mathrm{Im}(a_s+\sqrt{-1}c_s)^n}\\
&=\frac{n(\mathrm{Re}(a_s+\sqrt{-1}c_s)^{n-1}\wedge \cot(Q_I(B))b+\mathrm{Im}(a_s+\sqrt{-1}c_s)^{n-1}\wedge b)}{\mathrm{Im}(a_s+\sqrt{-1}c_s)^n}\\
&-\frac{\mathrm{Re}(a_s+\sqrt{-1}c_s)^n}{\mathrm{Im}(a_s+\sqrt{-1}c_s)^n}\frac{n(-\mathrm{Re}(a_s+\sqrt{-1}c_s)^{n-1}\wedge b+\mathrm{Im}(a_s+\sqrt{-1}c_s)^{n-1}\wedge \cot(Q_I(B)) b)}{\mathrm{Im}(a_s+\sqrt{-1}c_s)^n}\\
&=n(\cot(Q_I(B))+\cot(Q_{c_s}(a_s))) \cdot\\
&\frac{(\mathrm{Re}(a_s+\sqrt{-1}c_s)^{n-1}-\cot(Q_I(B)+Q_{c_s}(a_s))\mathrm{Im}(a_s+\sqrt{-1}c_s)^{n-1})\wedge b} {\mathrm{Im}(a_s+\sqrt{-1}c_s)^n}
>0
\end{split}\]
by Lemma 8.2 of \cite{CollinsJacobYau}. So $I$ is open. It must be $[0,1]$. So $1\in I$. It follows that
\[Q_I(A-CD\bar C^{T})=Q_{c_1}(a_1)\le Q_{c+b}(a)=Q_{I+CF\bar C^{T}}(A-CE\bar C^{T}).\]
Thus, we have proved that
\[Q_{I}(A-CD\bar C^{T})+Q_I(B)\le Q_I(\begin{bmatrix}
    A       & C \\
    \bar C^{T} & B
\end{bmatrix}).\]

As a special case, if $p=1$, then $A$ is a number and $A-CD\bar C^{T}\le A$. So
\[Q_I(B)\le Q_I(\begin{bmatrix}
    A       & C \\
    \bar C^{T} & B
\end{bmatrix})-Q_{I}(A-CD\bar C^{T})\le Q_I(\begin{bmatrix}
    A       & C \\
    \bar C^{T} & B
\end{bmatrix})-Q_{I}(A).\]
It is easy to see that \[Q_I(B)\le Q_I(\begin{bmatrix}
    A       & C \\
    \bar C^{T} & B
\end{bmatrix})-Q_{I}(A)=Q_I(\begin{bmatrix}
    A       & C \\
    \bar C^{T} & B
\end{bmatrix})-\arccot(A)\]
also holds if $B$ is only an Hermitian matrix which may be non-diagonal.
It follows that
\[P_\chi(\omega)=\max_{V\ \text{is a hyperplane}} Q_{\chi|_V}(\omega|_V),\]
for any $\chi>0$ and any $\omega$ satisfying $Q_\chi (\omega)<\pi$.
This is a generalization of the celebrated Courant–Fischer–Weyl min-max principle. As a corollary, we get the required estimate for $P$ by taking restrictions on hyperplanes.
\end{proof}

Now we need to prove the following:

\begin{proposition}
Let $\chi_{M\times M}=\pi_1^*\chi+\pi_2^*\chi$ be a K\"ahler form on $M\times M$. Let $C_{14}$, $\theta_{15}$, $\Theta_{16}$, $\Theta_{17}$, $\Theta_{18}$ be constants depending only on $n$, $\theta_0$, $\Theta_0$ such that
\[\theta_{15}=\theta_0+n\arccot(C_{14})<\Theta_{16}=\Theta_{17}+n\arccot(C_{14})<\Theta_{18}=\Theta_0+n\arccot(C_{14})<\pi.\]
Suppose that
\[\omega_{19}=\pi_1^*\omega_0+C_{14}\pi_2^*\chi+\sqrt{-1}\partial\bar\partial\varphi_{19}\]
is a real closed (1,1)-form on $M\times M$ such that $\omega_{19}\in\Gamma_{\chi_{M\times M}, \theta_{15},\Theta_{16}}$.
Define $\omega_{20}$ by \[\omega_{20}=\frac{\sum_{k=0}^{\lfloor\frac{n-1}{2}\rfloor}(-1)^k\frac{n!}{(n-2k)!(2k+1)!}(\pi_1)_*(\omega_{19}^{n-2k}\wedge\pi_2^*\chi^{2k+1})}{\int_M \mathrm{Im}(C_{14}\chi+\sqrt{-1}\chi)^n},\]
then $\omega_{20} \in \Gamma_{\chi, \theta_0, \Theta_0}$.
\label{Push-forward}
\end{proposition}
\begin{remark}
Proposition \ref{Push-forward} also holds when $\Theta_{16}=\Theta_{18}$. However, we require $\Theta_{16}<\Theta_{18}$ instead to make sure that if $\omega\in\Gamma_{\chi, \theta_{15}, \Theta_{16}}$, then $\frac{1}{\cot(Q_\chi(\omega))-\cot(\Theta_{18})}$ only changes a little bit when do the truncation in Section 3 of \cite{Chen}.
\label{Remark-push-forward}
\end{remark}

\begin{proof}
As in Section 5 of \cite{Chen}, at each $p=(p_1,p_2)\in M\times M$, let $z^{(1)}_i$ be the local coordinates on $M\times\{p_2\}$ and $z^{(2)}_i$ be the local coordinates on $\{p_1\}\times M$.
Then
\[\omega_{19}=\omega^{(1)}+\omega^{(2)}+\omega^{(1,2)}+\omega^{(2,1)},\]
where \[\omega^{(1)}=\sum_{i,j=1}^{n}\sqrt{-1}\omega^{(1)}_{i\bar j}dz_i^{(1)}\wedge d\bar z_j^{(1)}, \omega^{(2)}=\sum_{i,j=1}^{n}\sqrt{-1}\omega^{(2)}_{i\bar j}dz_i^{(2)}\wedge d\bar z_j^{(2)},\]
\[\omega^{(1,2)}=\sum_{i,j=1}^{n}\sqrt{-1}\omega^{(1,2)}_{i\bar j}dz_i^{(1)}\wedge d\bar z_j^{(2)}, \omega^{(2,1)}=\overline{\omega^{(1,2)}}.\]
After changing the definition of $z^{(2)}_i$ if necessary, we can assume that \[\pi_2^*\chi=\sqrt{-1}\sum_{i=1}^n dz^{(2)}_i\wedge d\bar z^{(2)}_i\] and \[\omega^{(2)}=\sqrt{-1}\sum_{i=1}^n \lambda_i dz^{(2)}_i\wedge d\bar z^{(2)}_i\] at $p$.
Then by the calculation in Section 5,
$\omega_{20}=(\pi_1)_*\omega_{21}$,
where \[\omega_{21}=\frac{\mathrm{Im}(\omega^{(2)}+\sqrt{-1}\pi_2^*\chi)^n}{\int_{\{p_1\}\times M}\mathrm{Im}(\omega^{(2)}+\sqrt{-1}\pi_2^*\chi)^n}\wedge
\sum_{i,j,l=1}^{n}(\omega^{(1)}_{i\bar j}-\omega^{(1,2)}_{i\bar l}\frac{\mathrm{Im}\prod_{k\not=l}(\lambda_k+\sqrt{-1})}{\mathrm{Im}\prod_{k=1}^{n}(\lambda_k+\sqrt{-1})}\overline{\omega^{(1,2)}_{j\bar l}})dz_i^{(1)}\wedge d\bar z_j^{(1)}.\]
By Lemma \ref{Subadditive-submatrix}, \[Q_{\pi_1^*\chi}(\sum_{i,j,l=1}^{n}(\omega^{(1)}_{i\bar j}-\omega^{(1,2)}_{i\bar l}\frac{\mathrm{Im}\prod_{k\not=l}(\lambda_k+\sqrt{-1})}{\mathrm{Im}\prod_{k=1}^{n}(\lambda_k+\sqrt{-1})}\overline{\omega^{(1,2)}_{j\bar l}})dz_i^{(1)}\wedge d\bar z_j^{(1)})<\Theta_{16}-Q_{\pi_2^*\chi}(\omega^{(2)}).\]

Now we consider the function $\frac{1}{\cot(Q_\chi(\omega))-\cot(\Theta_{18})}$ for $\omega\in\Gamma_{\chi,\theta_{15},\Theta_{16}}$.
Since \[D^2 (\frac{1}{\cot(Q_\chi(\omega))-\cot(\Theta_{18})})=\frac{-D^2 \cot(Q_\chi(\omega))}{(\cot(Q_\chi(\omega))-\cot(\Theta_{18}))^2}+\frac{2D\cot(Q_\chi(\omega)) \otimes D\cot(Q_\chi(\omega))}{(\cot(Q_\chi(\omega))-\cot(\Theta_{18}))^3}\]
and $\cot(Q_\chi(\omega))$ is concave by Lemma \ref{Property-of-F} (9), we see that $\frac{1}{\cot(Q_\chi(\omega))-\cot(\Theta_{18})}$ is convex on $\Gamma_{\chi,\theta_{15},\Theta_{16}}$.

So
\[\begin{split}
&\frac{1}{\cot(Q_\chi(\omega_{20}))-\cot(\Theta_{18})}\\
&<\int_{\{p_1\}\times M} \frac{1}{ \cot(\Theta_{18}-Q_{\pi_2^*\chi}(\omega^{(2)}))-\cot(\Theta_{18})}\frac{\mathrm{Im}(\omega^{(2)}+\sqrt{-1}\pi_2^*\chi)^n}{\int_{\{p_1\}\times M}\mathrm{Im}(\omega^{(2)}+\sqrt{-1}\pi_2^*\chi)^n}\\
&=\frac{\int_{\{p_1\}\times M} (\mathrm{Re}(\omega^{(2)}+\sqrt{-1}\pi_2^*\chi)^n-\cot(\Theta_{18})\mathrm{Im}(\omega^{(2)}+\sqrt{-1}\pi_2^*\chi)^n)} {(1+\cot^2(\Theta_{18}))\int_{\{p_1\}\times M}\mathrm{Im}(\omega^{(2)}+\sqrt{-1}\pi_2^*\chi)^n}\\
&=\frac{\mathrm{Re}(C_{14}+\sqrt{-1})^n-\cot(\Theta_{18})\mathrm{Im}(C_{14}+\sqrt{-1})^n} {(1+\cot^2(\Theta_{18}))\mathrm{Im}(C_{14}+\sqrt{-1})^n}\\
&=\frac{1}{\cot(\Theta_0)-\cot(\Theta_{18})}.
\end{split}\]
By a similar calculation, $\frac{1}{\cot(P_\chi(\omega_{20}))-\cot(\Theta_{18})}<\frac{1}{\cot(\theta_0)-\cot(\Theta_{18})}$. By the convexity of the set $\Gamma_{\chi,\theta_{15},\Theta_{16}}$ we also know that $\omega_{20}\in\Gamma_{\chi,\theta_{15},\Theta_{16}}$. It follows that $\omega_{20}\in \Gamma_{\chi, \theta_0, \Theta_0}$.
\end{proof}

It is easy to see that there exist constants $C_{21}>0$, $C_{22}>0$ depending only on $n,\theta_0,\Theta_0$ such that
\[\mathrm{Re}\prod_{k=1}^{n}(\lambda_k+\sqrt{-1})-\cot(\theta_0)\mathrm{Im}\prod_{k=1}^{n}(\lambda_k+\sqrt{-1})\le C_{21} \prod_{k=1}^{n}(\lambda_k-\cot(\Theta_{18}))\]
for all $\lambda\in\bar\Gamma_{\theta_{15},\Theta_{16}}$. We also know that $\omega-\cot(\Theta_{18})\chi$ is a K\"ahler form for all $\omega\in \Gamma_{\chi,\theta_{15},\Theta_{16}}$. Combining these facts with Proposition \ref{Solvability-assuming-sub-solution}, Remark \ref{Property-subsolution}, Proposition \ref{Push-forward} and Remark \ref{Remark-push-forward}, we can prove Proposition \ref{Current-sub-solution} using a similar method as Section 3 of \cite{Chen}.

\section{Finish of proof}

In this section, we finish the proof of Theorem \ref{Main-theorem}. If (1) of Theorem \ref{Main-theorem} holds, then for any smooth test family $\omega_{t,0}$, we can define another family $\omega_{t,\varphi}=\omega_{t,0}+\sqrt{-1}\partial\bar\partial\varphi$. Then $P_{\chi}(\omega_{t,\varphi})<P_{\chi}(\omega_\varphi)\le\theta_0$. So by Lemma 8.2 of \cite{CollinsJacobYau},
\[\begin{split}
&\frac{d}{dt}\int_V(\mathrm{Re}(\omega_{t,\varphi}+\sqrt{-1}\chi)^p-\cot(\theta_0)\mathrm{Im}(\omega_{t,\varphi}+\sqrt{-1}\chi)^p)\\
&=\int_V p(\mathrm{Re}(\omega_{t,\varphi}+\sqrt{-1}\chi)^{p-1}-\cot(\theta_0)\mathrm{Im}(\omega_{t,\varphi}+\sqrt{-1}\chi)^{p-1})\wedge \frac{d}{dt}\omega_{t,\varphi}\ge0.
\end{split}\]
By Lemma 8.2 of \cite{CollinsJacobYau}, we also know that there exists a constant $\epsilon_1>0$ such that for any point $x\in M$ and any $p$-dimensional vector space $V_x\subset T_x M$, the restriction of the form
\[\mathrm{Re}(\omega_{t,0}+\sqrt{-1}\chi)^p-\cot(\theta_0)\mathrm{Im}(\omega_{t,0}+\sqrt{-1}\chi)^p-(n-p)\epsilon_1 \chi^p\]
on $V_x$ is positive. Then we get (2) of Theorem \ref{Main-theorem} of using the fact that
\[\begin{split}
&\int_V(\mathrm{Re}(\omega_{t,0}+\sqrt{-1}\chi)^p-\cot(\theta_0)\mathrm{Im}(\omega_{t,0}+\sqrt{-1}\chi)^p)\\
&=\int_V(\mathrm{Re}(\omega_{t,\varphi}+\sqrt{-1}\chi)^p-\cot(\theta_0)\mathrm{Im}(\omega_{t,\varphi}+\sqrt{-1}\chi)^p).
\end{split}\]

It is trivial that (2) of Theorem \ref{Main-theorem} implies (3) of Theorem \ref{Main-theorem}. On the other hands, as long as Proposition \ref{Main-induction} holds, then (3) of Theorem \ref{Main-theorem} implies (1) of Theorem \ref{Main-theorem} by choosing $f=0$ and choosing arbitrary $\Theta_0\in(\theta_0,\pi)$.

Therefore, it suffices to prove Proposition \ref{Main-induction}. We prove it by induction on the dimension $n$ of $M$. When $n=1$, it is trivial. So we assume that it has been proved for all smaller dimensions and then try to prove it. Define $I$ be the set of $t\in[0,\infty)$ such that there exists a smooth function $\varphi_t$ and a constant $c_t\ge 0$ satisfying $\omega_{t,\varphi_t}=\omega_{t,0}+\sqrt{-1}\partial\bar\partial\varphi_t\in\Gamma_{\chi,\theta_0,\Theta_0}$ and
\[\mathrm{Re}(\omega_{t,\varphi_t}+\sqrt{-1}\chi)^n-\cot(\theta_0)\mathrm{Im}(\omega_{t,\varphi_t}+\sqrt{-1}\chi)^n-c_t\chi^n=0.\]
By (C) of the definition of test family and Proposition \ref{Solvability-assuming-sub-solution}, $I$ is non-empty. We also know that $I$ is open by Proposition \ref{Solvability-assuming-sub-solution}. To show the closeness, assume that $t_k\in I$ is a sequence converging to $t_\infty$. Then by monotonicity of $P_\chi$ and Proposition \ref{Solvability-assuming-sub-solution}, we know that $t\in I$ for all $t>t_\infty$. We need to show that $t_\infty\in I$. Without loss of generality, assume that $t_\infty=0$. By Proposition \ref{Current-sub-solution}, there exist a constant $\epsilon_3>0$ and a current $\omega_4=\omega_0-\epsilon_3\chi+\sqrt{-1}\partial\bar\partial\varphi_4$ such that $\omega_4\in \bar\Gamma_{\chi,\theta_0,\Theta_0}$ in the sense of Definition \ref{Definition-current-sub-solution}. By Proposition \ref{Solvability-assuming-sub-solution}, it suffices to find $\omega_{23}=\omega_0+\sqrt{-1}\partial\bar\partial\varphi_{23}\in\Gamma_{\chi,\theta_0,\Theta_0}$. We essentially follow the works in \cite{Chen}, with minor adjustments to deal with the problem that $\omega_0$ is no longer K\"ahler.

Let $\epsilon_{12}=\min\{\frac{\epsilon_3}{3},\frac{1}{100}\}$, then we get the corresponding $\epsilon_{13}$. By choosing $\epsilon_{13}$ small enough, we can also assume that as long as $(1+\epsilon_{13})\chi_0\ge\chi\ge\chi_0$ and $\omega\in \bar\Gamma_{\chi_0,\theta_0,\Theta_0}$, then $\omega+\epsilon_{12}\chi\in \Gamma_{\chi,\theta_0,\Theta_0}$. We can also assume that $\epsilon_{13}<\frac{1}{100}$. Then similar to Section 4 of \cite{Chen}, there exists a finite number of coordinate balls $B_{2r}(x_i)$ such that $B_r(x_i)$ is a cover of $M$. Moreover,let $\varphi^i_{\omega_0}, \varphi^i_{\chi}$ be potentials such that
$\sqrt{-1}\partial\bar\partial\varphi^i_{\omega_0}=\omega_0$ and $\sqrt{-1}\partial\bar\partial\varphi^i_{\chi}=\chi$ on $B_{2r}(x_i)$, then we can also required that
\[|\varphi^i_{\chi}-|z|^2|\le \frac{\epsilon_{12}r^2}{100(1+|\cot(\theta_0)|)}\]
and
\[\sqrt{-1}\partial\bar\partial|z|^2\le \chi\le (1+\epsilon_{13})\sqrt{-1}\partial\bar\partial|z|^2 \]
on $B_{2r}(x_i)$.
By the uniform continuity of $\varphi^i_{\omega_0}$ there exists $\epsilon_{24}<\frac{r}{5}$ such that $|\varphi^i_{\omega_0}(x)-\varphi^i_{\omega_0}(y)|\le \frac{\epsilon_{12}r^2}{100}$ for all $x\in \overline{B_{\frac{9r}{5}}(x_i)}$ such that $|x-y|\le\epsilon_{24}$.

As in \cite{Chen}, we take $\delta<\frac{\epsilon_{24}\epsilon_{12}}{100(1+|\cot(\theta_0)|)}$ and let $\varphi_\delta^i$ be the smoothing of $\varphi^i_{\omega_0}-2\epsilon_{12}\varphi^i_{\chi}+\varphi_4$, and let $\varphi_{25}^i=\varphi^i_{\delta}-\varphi^i_{\omega_0}+\epsilon_{12}\varphi^i_{\chi}$, then we know that $\omega_0+\sqrt{-1}\partial\bar\partial\varphi_{25}^i\in\Gamma_{\chi,\theta_0,\Theta_0}$ on $\overline{B_{\frac{9r}{5}}(x_i)}$. We also know that $\omega_0-3\epsilon_{12}\chi+\varphi_4-\cot(\theta_0)\chi$ is a positive current.

Similar to \cite{Chen}, we pick a small enough number
\[\epsilon_{26}=\frac{\epsilon_{12} r^2}{100(\int_0^1\log(\frac{1}{t})\mathrm{Vol}(\partial B_1(0))t^{2n-1}\rho(t)dt+\log 2+\frac{3^{2n-1}}{2^{2n-3}}\log 2)},\] where $\rho$ is function in Definition \ref{Definition-smoothing}. Then we consider the set $Y$ where the Lelong number of $\varphi_4$ is at least $\epsilon_{26}$. By Siu's work \cite{Siu}, $Y$ is a subvariety.

If we can find a smooth function $\varphi_{27}$ near $Y$ such that $\omega_0+\sqrt{-1}\partial\bar\partial\varphi_{27}\in\Gamma_{\chi,\theta_0,\Theta_0}$, then using the methods in \cite{Chen}, as long as $\delta$ is small enough, the regularized maximum of $\varphi_{27}+3\epsilon_{26}\log\delta$ with $\varphi_{25}^i$ provides the required smooth function $\varphi_{23}$ on $M$.

Therefore, we only need to find $\varphi_{27}$. By Hironaka's desingularization theorem, there exists a blow-up $\tilde M$ of $M$ obtained by a sequence of blow-ups with smooth centers such that the proper transform $\tilde Y$ of $Y$ is smooth. Without loss of generality, assume that we only need to blow up once because otherwise we just repeat the process. Let $\pi$ be the projection from $\tilde M$ to $M$. Let $E$ be the exceptional divisor. Let $s$ be the defining section of $E$. Let $h$ be any smooth metric on the line bundle $[E]$, then $\frac{\sqrt{-1}}{2\pi}\partial\bar\partial\log|s|_h^2=[E]+\omega_{28}$ by the Poincar\'e-Lelong equation. Then there exists a constant $C_{29}$ such that the $\omega_{30}=\omega_{28}+C_{29}\pi^*\chi>0$.

Let $\omega_{2,0}$ be $\omega_{t,0}$ when $t=2$ and $\omega_{1,0}$ be $\omega_{t,0}$ when $t=1$. Then there exists a constant $\epsilon_{31}>0$ such that $\omega_{2,0}-\omega_{1,0}\ge\epsilon_{31}\chi$. There exists a smooth function $\varphi_{32}$ on $M$ such that $\omega_{1,0}+\sqrt{-1}\partial\bar\partial\varphi_{32}\in\Gamma_{\chi,\theta_0,\Theta_0}$ on $M$. It implies that
\[\begin{split}
\int_V(\mathrm{Re}(\omega_{t,0}-\epsilon_{31}\chi+\sqrt{-1}\chi)^p-\cot(\theta_0)\mathrm{Im}(\omega_{t,0}-\epsilon_{31}\chi+\sqrt{-1}\chi)^p)\\
\ge\int_V(\mathrm{Re}(\omega_{1,0}+\sqrt{-1}\chi)^p-\cot(\theta_0)\mathrm{Im}(\omega_{1,0}+\sqrt{-1}\chi)^p)\ge (n-p)\epsilon_1\int_V \chi^p.
\end{split}\]
for all $t\ge 2$ and all $p$-dimensional subvariety $V$ of $M$. By choosing $\epsilon_{31}$ small enough, using the fact that $\omega_{t,0}$ is bounded with respect to $\chi$ for all $t\in[0,2]$, we can also assume that
\[\int_V(\mathrm{Re}(\omega_{t,0}-\epsilon_{31}\chi+\sqrt{-1}\chi)^p-\cot(\theta_0)\mathrm{Im}(\omega_{t,0}-\epsilon_{31}\chi+\sqrt{-1}\chi)^p)\ge (n-p)\frac{\epsilon_1}{2}\int_V \chi^p.\]
for all $t\ge 0$ and all $p$-dimensional subvariety $V$ of $M$.

Now we want to find constants $0<\epsilon_{33}<1$ and $C_{34}$ independent of $t$ and consider the K\"ahler form $\pi^*\chi+\epsilon_{33}\omega_{30}$ on $\tilde M$ and the test family \[\pi^*\omega_{t,0}-\epsilon_{31}\pi^*\chi+(t+C_{34})\epsilon_{33}\omega_{30}\] on $\tilde M$. We know that $\pi(E)$ is smooth. So by induction hypothesis, as in \cite{Chen}, we can find a smooth function $\varphi_{35}$ on $M$ such $\omega_{35}=\omega_0+\sqrt{-1}\partial\bar\partial\varphi_{35}$ satisfies $\omega_{35}\in \Gamma_{\chi,\theta_0,\Theta_0}$ on a neighborhood $U_{36}$ of $\pi(E)$. By shrinking $U_{36}$ and replacing $\epsilon_{31}$ if necessary, we can assume that there exists a constant $\epsilon_{37}>0$ such that $\omega_{35}-\epsilon_{31}\chi\in \Gamma_{\chi,\theta_0-\epsilon_{37},\Theta_0}$ on $U_{36}$. By compactness, there exists a constant $\epsilon_{38}>0$ such that $\omega_{1,0}+\sqrt{-1}\partial\bar\partial\varphi_{32}\in\Gamma_{\chi,\theta_0-\epsilon_{38},\Theta_0}$ on $M$.

Then we required that $C_{34}>\cot(\frac{\epsilon_{37}}{n})$ and $C_{34}>\cot(\frac{\epsilon_{38}}{n})$. By Lemma 8.2 of \cite{CollinsJacobYau}, when $t\ge 2$,
\[\begin{split}
&\mathrm{Im}(e^{-\sqrt{-1}\theta_0}(\pi^*\omega_{t,32}+(t+C_{34})\epsilon_{33}\omega_{30}
+\sqrt{-1}(\pi^*\chi+\epsilon_{33}\omega_{30}))^q)\\
-&\mathrm{Im}(e^{-\sqrt{-1}\theta_0}(\pi^*\omega_{t,32}
+\sqrt{-1}\pi^*\chi)^q)\\
=&\mathrm{Im}(e^{-\sqrt{-1}\theta_0}\sum_{k=0}^{q-1}\frac{q!}{k!(q-k)!}(\pi^*\omega_{t,32}
+\sqrt{-1}\pi^*\chi)^k
\wedge((t+C_{34}+\sqrt{-1})\epsilon_{33}\omega_{30})^{q-k})\\
\le&\mathrm{Im}(e^{-\sqrt{-1}\theta_0}(t+C_{34}+\sqrt{-1})^q\epsilon_{33}^q\omega_{30}^q),
\end{split}\]
where $\omega_{t,32}=\omega_{t,0}-\epsilon_{31}\chi+\sqrt{-1}\partial\bar\partial\varphi_{32}$.
So
\[\begin{split}
&\int_V\mathrm{Im}(e^{-\sqrt{-1}\theta_0}(\pi^*\omega_{t,0}-\epsilon_{31}\pi^*\chi+(t+C_{34})\epsilon_{33}\omega_{30}
+\sqrt{-1}(\pi^*\chi+\epsilon_{33}\omega_{30}))^q)\\
\le&\int_V\mathrm{Im}(e^{-\sqrt{-1}\theta_0}(\pi^*\omega_{t,0}-\epsilon_{31}\pi^*\chi
+\sqrt{-1}\pi^*\chi)^q)
+\int_V\mathrm{Im}(e^{-\sqrt{-1}\theta_0}(t+C_{34}+\sqrt{-1})^q\epsilon_{33}^q\omega_{30}^q)\\
\le&-\sin(\theta_0)(n-q)\frac{\epsilon_1}{2}\int_V\pi^*\chi^q+\mathrm{Im}(e^{-\sqrt{-1}\theta_0}(t+C_{34}+\sqrt{-1})^q)\int_V (\epsilon_{33}\omega_{30})^q\\
\le&-\epsilon_{39}\int_V(\pi^*\chi+\epsilon_{33}\omega_{30})^q,
\end{split}\]
for any $q$-dimensional subvariety $V$ of $\tilde M$ and a constant $\epsilon_{39}$ independent of $t$ and $V$.

On the other hands, for $t\in[0,2]$, we get a similar estimate within $U_{36}$. As for the set $\tilde M\setminus\pi^{-1}(U_{36})$, we know that all the forms $\pi^*\omega_{35}$, $\pi^*\chi$, $\omega_{28}$ are all bounded using the norm defined by $\pi^*\chi+\epsilon_{33}\omega_{30}$ and moreover, $\pi^*\chi$ is also bounded below by positive constant multiple of $\pi^*\chi+\epsilon_{33}\omega_{30}$. So if $\epsilon_{33}$ is small enough, then the K\"ahler form $\pi^*\chi+\epsilon_{33}\omega_{30}$ and the test family $\pi^*\omega_{t,0}-\epsilon_{31}\pi^*\chi+(t+C_{34})\epsilon_{33}\omega_{30}$ on $\tilde M$ satisfies the assumption of Proposition \ref{Main-induction}. Since $E$ is smooth there, by induction hypothesis and the arguments in \cite{Chen}, there exists a smooth function $\varphi_{40}$ on $\tilde M$ such that
\[\pi^*\omega_{0}-\epsilon_{31}\pi^*\chi+C_{34}\epsilon_{33}\omega_{30}+\sqrt{-1}\partial\bar\partial\varphi_{40}
\in\Gamma_{\pi^*\chi+\epsilon_{33}\omega_{30},\theta_0,\Theta_0}\] on a neighborhood $U_{41}$ of $\pi^{-1}(Y)$. By a similar argument as in the proof of Lemma \ref{Subadditive-submatrix}, it implies that
\[\pi^*\omega_{0}-\epsilon_{31}\pi^*\chi+(C_{34}-\cot(\theta_0))\epsilon_{33}\omega_{30}+\sqrt{-1}\partial\bar\partial\varphi_{40}
\in\Gamma_{\pi^*\chi,\theta_0,\Theta_0}\] on $U_{41}\setminus \pi^{-1}(E)$.
So by choosing $(C_{34}-\cot(\theta_0))\epsilon_{33}C_{29}<\epsilon_{31}$, we see that
\[\pi^*\omega_{0}+\sqrt{-1}\partial\bar\partial(\varphi_{40}+(C_{34}-\cot(\theta_0))\epsilon_{33}\frac{\sqrt{-1}}{2\pi}\log|s|_h^2)
\in\Gamma_{\pi^*\chi,\theta_0,\Theta_0}\] on $U_{41}\setminus \pi^{-1}(E)$. Finally, we choose a large enough constant $C_{42}$ and define $\varphi_{27}$ as the regularized maximum of $\pi_*(\varphi_{40}+(C_{34}-\cot(\theta_0))\epsilon_{33}\frac{\sqrt{-1}}{2\pi}\log|s|_h^2)$ with $\varphi_{35}-C_{42}$.

\bibliographystyle{amsalpha}

\bibliography{dHYM}

 \end{document}